\documentclass[leqno,11pt]{article}

\usepackage[usenames, dvipsnames]{xcolor}
\usepackage{tikz}

\usepackage{mathpazo}
\usepackage{amssymb,hyperref,amsthm}
\usepackage{euscript}
\usepackage{flafter}
\usepackage{pstricks}
\usepackage[latin1]{inputenc}
\usepackage{amsmath}
\usepackage{amsfonts}
\usepackage{pstricks-add}
\usepackage{yfonts}
\usepackage{egothic}
\usepackage{dsfont}
\usepackage{bm}
\usepackage{variations}

\hypersetup{
    linktoc=page,
   linkcolor=red,          
  citecolor=blue,        
    filecolor=blue,      
   urlcolor=cyan,
    colorlinks=true           
}

\usepackage[refpage]{nomencl}

\usepackage{makeidx}
\makeindex

\usepackage{enumitem}

\usepackage{mathrsfs} 

\evensidemargin\oddsidemargin

\DeclareFontFamily{U}{mathx}{\hyphenchar\font45}
\DeclareFontShape{U}{mathx}{m}{n}{
      <5> <6> <7> <8> <9> <10>
      <10.95> <12> <14.4> <17.28> <20.74> <24.88>
      mathx10
      }{}
\DeclareSymbolFont{mathx}{U}{mathx}{m}{n}
\DeclareMathAccent{\widecheck}{\mathalpha}{mathx}{"71}

\newcommand{\eqnsection}{
\renewcommand{\theequation}{\thesection.\arabic{equation}}
   \makeatletter
   \csname  @addtoreset\endcsname{equation}{section}
   \makeatother}
\eqnsection

\newtheorem{hypo}{Hypothesis}

\newtheorem{thm}[hypo]{Theorem}

\newtheorem{lem}[hypo]{Lemma}

\newtheorem{rqe}[hypo]{Remark}

\def\E{\mathcal{E}}

\def\PP{\mathbb{P}}

\def\EE{\mathbb{E}}
\def\NN{\mathbb{N}}

\let\BFseries\bfseries\def\bfseries{\BFseries\mathversion{bold}} 

\title{Branching processes in correlated random environment}

\author{XINXIN CHEN\footnote{Institut Camille Jordan - C.N.R.S. UMR 5208 - Universit\'e Claude Bernard Lyon 1
(France).
 \newline \vspace{0.1cm} \hspace{0.2cm} $\dag$Institut Camille Jordan - C.N.R.S. UMR 5208 - Universit\'e Claude Bernard Lyon 1
(France). 
\newline \vspace{0.1cm}    MSC 2000  60J80 60K37 60G15. Supported by ANR MALIN ANR-16-CE93-0003\newline \vspace{0.5cm} \textit{Key words :  branching process, correlated random environment} } $\ $,$\ $ NADINE GUILLOTIN-PLANTARD$^\dag$  }

\begin{document}
\baselineskip=17pt
\setcounter{page}{1}

\maketitle

We consider the critical branching processes in correlated random environment which is positively associated and study the probability of survival up to the $n$-th generation. Moreover, when the environment is given by fractional Brownian motion, we estimate also the tail of progeny as well as the tail of width.

\section{Introduction and results}

In the theory of branching process, branching processes in random environment (BPRE), as an important part, was introduced by Smith and Wilkinson \cite{SW69} by supposing that the environment is i.i.d.. This model has been well investigated by lots of authors. One can refer to \cite{Afa},\cite{AK1},\cite{AK2} for various properties obtained in this setting. In fact, for this so called Smith-Wilkinson model, the behaviour of BPRE depends largely on the behaviour of the corresponding random walk constructed by the logarithms of the quenched expectation of population sizes. As this random walk is of i.i.d. increments due to i.i.d. environment, many questions on this model become quite clear. 

However, we are interested in branching processes in correlated random environment. More precisely, we consider the Athreya-Karlin model of BPRE where the environment is assumed to be stationary and ergodic; and moreover correlated. 

Let us introduce some notations. Consider a branching process $(Z_n)_{n\geq 0}$ in random environment given by a sequence of random generating functions $\E=\{f_0,f_1,\ldots, f_n,\ldots\}$. 
Given the environment, individuals reproduce independently of each other. The offspring of an individual in the n-th generation has generating function $f_n$. If $Z_n$ denotes the number of individuals in the n-th generation, then under the quenched probability $\PP^\E$ (and the quenched expectation $\EE^\E$),
\[
\EE^\E[s^{Z_{n+1}}\vert Z_0,\cdots,Z_n]=(f_n(s))^{Z_n}, \forall n\geq 0.
\]
We will assume that $Z_0=1$. Here the random environment $\{f_n; n\geq0\}$ is supposed to be stationary, ergodic and correlated. 
The process $(Z_n)_{n\geq 0}$ will be called {\it a branching process in correlated environment} (BPCE, for short).

First of all, the criterion for the process to be subcritical, critical or supercritical was proven by Tanny \cite{Tanny}. In this paper, we only consider the non-sterile critical case, i.e.
\begin{equation}
\EE[\log f_0'(1)]=0, \PP^\E(Z_1=1)<1,
\end{equation}
where $\EE(\cdot)$ is the annealed expectation.

We are interested in some important quantities related to this branching process, such as the tail distribution of its extinction time $T$,
of its maximum population and of its total population size: 
\[
\PP(T>n),\ \PP\Big(\max_{j\geq 0}Z_j > N\Big),\ \PP\Big(\sum_{j\geq0}Z_j >N\Big).
\]
Let us mention that this problem was considered in \cite{AGP} in the case where the offspring sizes are 
geometrically distributed, using the well-known correspondence between
recurrent random walks in random environment and critical branching processes in random environment with geometric distribution of offspring sizes. Our aim is to generalise the results obtained
in \cite{AGP} to more general generating functions $(f_n)_{n\geq 0}$.

More precisely let $X_{i+1} := - \log (f_{i}'(1))$ for every $i\geq 0$. Assume that $(X_i)_{i\geq 1}$ is a stationary, ergodic and centered sequence and define the sequence ($S_0=0$)
$$S_n:=\sum_{i=1}^n X_i\  \mbox{\rm for}\ n\geq 1.$$
We also assume that the scaling limit of $(S_n)_ {n\ge 0}$ is a stochastic process $(W(t))_{t\geq 0}$:
\begin{equation} \label{eqn:weakconvergencetofbm}
\left( n^{-H} \ell(n)^{-1/2} S_{[ nt ]}\right)_{t\geq 0}   \mathop{\Longrightarrow}_{n\rightarrow\infty}^{\mathcal{L}} \left( W(t)\right)_{t\geq 0},
\end{equation}
where $H\in (0,1)$ and $\ell$ is a slowly varying function at infinity such that as $n\rightarrow\infty$ 
\begin{equation} \label{momentordre2}
\EE[S_n^2] \sim n^{2H} \ell(n) \EE[ W^2(1)].
\end{equation}
We will also assume that the tail distribution of the random variable $X_1$ decreases sufficiently fast, namely there exist $\alpha\in (1,+\infty)$ and  $\gamma\in (0,+\infty)$
such that 
\begin{equation} \label{taildistribution}
\lim_{x\rightarrow +\infty} x^{-\alpha} \log \PP[ |X_1| \geq x] = -\gamma.
\end{equation}

Let us recall that a collection $\{Z_1,\ldots,Z_n\}$ of random variables defined on a same probability space is said {\it quasi-associated} provided that 
$$\mbox{Cov} \Big(f(Z_1,\ldots,Z_i) , g(Z_{i+1},\ldots, Z_n) \Big) \geq 0$$
for any $i=1,\ldots,n-1$ and all coordinatewise nondecreasing, measurable functions $f:\mathbb{R}^i\rightarrow \mathbb{R}$
and $g:\mathbb{R}^{n-i}\rightarrow \mathbb{R}$. 
We will say that $\{Z_1,\ldots,Z_n\}$ is {\it positively associated} if 
$$\mbox{Cov} \Big(f(Z_1,\ldots,Z_n) , g(Z_{1},\ldots, Z_n) \Big) \geq 0$$
for all coordinatewise nondecreasing, measurable functions $f,g :\mathbb{R}^n \rightarrow \mathbb{R}$. 
We refer to \cite{EPW} for details concerning positively associated random variables. Clearly positive association is a stronger assumption than quasi-association. A sequence of random variables $(Z_i)_{i\geq 1}$ is said {\it positively associated} (resp. {\it quasi-associated}) if for every $n\geq 2$, the set 
$\{Z_1,\ldots,Z_n\}$ is {\it positively associated} (resp. {\it quasi-associated}).


For every $i\geq 0$, we denote by $\sigma^2(f_i)$ the variance of the probability distribution with generating function $f_i$. Remark that $\sigma^2(f_i) = f_i''(1)+f_i'(1)-(f_i'(1))^2$. Our main assumption concerning the sequence $(\sigma^2(f_n))_{n\geq 0}$ is the following one:\\*
${\bf Assumption (A)}$ There exist positive constants $A,B$ and $C$ such that for every $i\geq 0$,
\[
 \sigma^2(f_i) \leq A(f'_i(1))^2+B f'_i(1)+C.
\]

Remark that the assumption ${\bf (A)}$ is satisfied for the classical discrete probability distributions such as the Poisson distribution, the Geometric distribution, the uniform distribution, the Binomial distribution etc.\\*
In this setup we obtain the following theorem.
\begin{thm}\label{main1}
Assume that the sequence $(X_i)_{i\geq 1}$ is positively associated. Under assumption ${\bf (A)}$, there exist positive constants $c,C$ such that for large enough $n$,
$$  n^{-(1-H)} \sqrt{\ell(n)} (\log n)^{-c}\leq \PP\Big[ T > n \Big] \leq C n^{-(1-H)} \sqrt{\ell(n)}.$$
\end{thm}
\begin{rqe} 
Actually we will prove that the upper bound holds for every $n\geq 1$. This is due to the fact that we use strong results on the persistence of the random walk $(S_n)_n$ namely Theorem 11 in \cite{AGPP}.
\end{rqe}
From now on we will assume that $(X_i)_{i\geq 1}$ is a standard Gaussian sequence with positive correlations $r(j):=\EE [X_0 X_j] = \EE[ X_k X_{j+k} ]$ satisfying as $n\rightarrow +\infty$,
\begin{equation} \label{eqn:lrd}
\sum_{i,j=1}^n r(i-j) = n^{2H} \ell(n), 
\end{equation}
where $H\in (0,1)$ and $\ell$ is a slowly varying function at infinity. In that case the process $(W(t))_ {t\ge 0}$ is the fractional Brownian motion $B_H$ with Hurst parameter $H$ (see \cite{taqqu}, \cite[Theorem 4.6.1]{Whitt}). Recall that $B_H$ is the real centered Gaussian process with
covariance function
\[
\mathbb E[B_H(t)B_H(s)]=\frac 12(t^{2H}+s^{2H}-|t-s|^{2H}).
\]
When $H\geq 1/2$, the sequence $(X_i)_{i\geq 1}$ is positively associated as positively correlated Gaussian random variables.
\begin{thm}\label{main2}
Under assumption ${\bf (A)}$, there exists a function $L$ that is slowly varying at infinity such that for large enough $N$
$$  \frac{(\log N)^{-\frac{(1-H)}{H}}}{L(\log N)} \leq \PP\Big[  \max_{k\ge 0} Z_k>N \Big] \leq (\log N)^{-\frac{(1-H)}{H}}L(\log N).$$
\end{thm}
\noindent Note that
\[
\max_{j\geq 0}Z_j\leq \sum_{j\geq0}Z_j\leq T\max_{j\geq0}Z_j.
\]
As a consequence, 
\[
\PP\Big(\sum_{j\geq0}Z_j> N^2\Big)-\PP(T>N)\leq \PP(\max_{j\geq0}Z_j>N)\leq \PP\Big(\sum_{j\geq0}Z_j>N\Big).
\]
so Theorems \ref{main1} and \ref{main2} lead to the following result.
\begin{thm}
Under assumption ${\bf (A)}$, there exists a function $L$ that is slowly varying at infinity such that for large enough $N$
$$    \frac{(\log N)^{-\frac{(1-H)}{H}}}{L(\log N)} \leq \PP\Big[  \sum_{k\ge 0} Z_k>N \Big] \leq  (\log N)^{-\frac{(1-H)}{H}}L(\log N).$$
\end{thm}

\section{Extinction time: Proof of Theorem \ref{main1}}
\subsection{Upper bound}
Observe that for any $m\leq n$,
\[
\PP^\E(T>n)\leq \PP^\E(T>m)=\PP^\E(Z_m\geq1)\leq \EE^\E[Z_m]=e^{-S_m}.
\]
Then,
\begin{equation}\label{upbT}
\PP(T>n)\leq \EE[e^{-\max_{0\leq m\leq n}S_m}]=\int_0^\infty e^{-x}\PP(\max_{0\leq m\leq n}S_m\leq x)dx
\end{equation}
as $\max_{0\leq m\leq n}S_m\geq0$. Let us bound $\PP(\max_{1\leq m\leq n}S_m\leq x)$ for $x>0$. Let $S_n^*:=\max_{1\leq m\leq n}S_m$. 
Note that for every $K\geq 0$,
\begin{align}\label{perSx}
\PP(S_n^*\leq 0)\geq & \PP(S_{n+K}^*\leq 0)\nonumber\\
\geq& \PP(\max_{1\leq j\leq K-1} S_j \leq 0; S_K\leq -x; \max_{1\leq j\leq n}(S_{K+j}-S_K)\leq x)\nonumber\\
\geq &\PP(\max_{1\leq j\leq K-1} S_j \leq 0; S_K\leq -x)\PP(\max_{1\leq j\leq n}(S_{K+j}-S_K)\leq x)
\end{align}
by quasi-association of $\{S_k; 1\leq k\leq n\}$. Note that by stationarity, we have
\[
\PP(\max_{1\leq j\leq n}(S_{K+j}-S_K)\leq x)= \PP(S_n^*\leq x).
\]
On the other hand, by positive association,
\[
\PP(\max_{1\leq j\leq K-1} S_j \leq 0; S_K\leq -x)\geq \PP(S_{K}^*\leq 0)\PP(S_K\leq -x).
\]
So,
\[
\PP(S_n^*\leq x) \PP(S_{K}^*\leq 0)\PP(S_K\leq -x)\leq \PP(S_n^*\leq 0).
\]

Let us prove that the sequence $(S_n)_{n\geq 1}$ satisfies the hypotheses of Theorems 2 and 4 in \cite{AGPP}.
Due to the convergence in law of $((n^{-H} \ell(n)^{-1/2} S_{\lfloor nt\rfloor})_t)$ to $(W(t))_t$  as $n$ goes to infinity,
we can show that for any $p\in(1,2)$ fixed, $(n^{-pH} l(n)^{-p/2} |\max_{1\le k\le n}S_k|^p)$ converges in distribution
to $(\sup_{t\in[0,1]} W(t))^p $ 
as $n$ goes to infinity (see Section 12.3 in \cite{Whitt}).
Let us prove that $(n^{-pH} \ell(n)^{-p/2} |S_n^*|^p)_{n\ge 1}$ is
uniformly integrable. 
To this end we will use the fact that the increments of $(S_n)_ {n\ge 1}$ are centered and positively associated.
Due to Theorem 2.1 of \cite{Gong}, there exists some constant $C_1>0$ such that
\[
\mathbb E\left[\left|S_n^*\right|^2\right] \leq  \mathbb E\left[\max_{j=1,\ldots,n} |S_j|^2 \right]   
 \le C_1\  \mathbb E\left[S_n^2 \right],
\]
The uniform integrability of $(n^{-pH} \ell(n)^{-p/2} |S_n^*|^p)_{n\ge 1}$ follows from assumption (\ref{momentordre2}), and then $\EE[ |S_n^*|^p] \sim n^{pH}\ell(n)^{p/2} \EE[ (\sup_{t\in[0,1]}W(t))^p]$ as $n$ goes to infinity.
From Theorem 4 in \cite{AGPP}, there exists $c_1>0$ such that for every
$K\geq 2$, 
$$\PP(S_{K}^*\leq 0)\geq c_1 \frac{K^{-(1-H)}}{\log K}\sqrt{l(K)}.$$
Moreover, if we choose $x=K^Hl(K)$, the probability $\PP(\frac{S_K}{x}\leq -1)$ converges to $\PP(W(1)\leq -1)\in(0,1)$. So, we get that (here $c,C$ are two positive constants) for $x$ large and for any $n\geq 1$
$$ \PP(S_n^*\leq x) \leq c_2\  \frac{\log x}{\tilde \ell(x)}\ x^{\frac{1}{H}-1}\PP(S_n^*\leq 0)$$
where $\tilde \ell$ is a slowly varying function at infinity.
Then, from the upper bound in Theorem 2 in \cite{AGPP}, we get that for $x$ large and for any $n\geq 1$,
$$ \PP(S_n^*\leq x)\leq c_3\ \frac{\log x}{\tilde \ell(x)}\ x^{\frac{1}{H}-1} n^{-(1-H)}\sqrt{\ell(n)}.$$
Plugging this into \eqref{upbT} implies that there exists $C_2>0$ such that for every $n\geq 1$,
\[
\PP(T>n)\leq C_2\ n^{-(1-H)}\sqrt{\ell(n)}.
\]

\subsection{Lower bound} Note that (see (2.1) in \cite{GK})
\[
\PP^\E(T>n)=\PP^\E(Z_n\geq1)=1-f_0\circ f_1\circ\cdots\circ f_{n-1}(0).
\]
It is known in \cite{GK} that
\[
\frac{1}{1-f_0\circ f_1\circ\cdots\circ f_{n-1}(0)}=\prod_{i=0}^{n-1}f_i'(1)^{-1}+\sum_{k=0}^{n-1}\prod_{i=0}^{k-1}f'_i(1)^{-1}\times \eta_{k,n}
\]
where
\[
\eta_{k,n}=: g_k(f_{k+1}\circ \cdots\circ f_{n-1}(0))
\]
with
\[
g_k(s)=\frac{1}{1-f_k(s)}-\frac{1}{f_k'(1)(1-s)}
\]
From Lemma 2.1 in \cite{GK},
$$
\eta_{k,n}\leq \frac{f_k''(1)}{f'_k(1)^2}=\frac{\sigma^2(f_k)+f'_k(1)^2-f'_k(1)}{f'_k(1)^2}.
$$
This yields that
\begin{align}
&\PP(T>n)\geq \EE\left[\frac{1}{1+\sum_{k=0}^{n-1}\sigma^2(f_k)e^{S_{k+1}+X_{k+1}}}\right]\nonumber\\
\geq & \EE\left[\frac{1}{1+A+(A+B)\sum_{k=1}^n e^{S_k}+C\sum_{k=1}^{n}e^{S_{k}+X_{k}}}\right]\hspace{1cm}\mbox{\rm from Assumption \bf (A)}\nonumber\\
\geq & \EE\left[\frac{1}{1+A+(A+B)\alpha_n e^{S^*_{\alpha_n}}+ (A+B) n e^{\max_{\alpha_n+1\leq j\leq n}S_j}+C\alpha_n e^{S_{\alpha_n}^*+X_{\alpha_n}^*}+C n e^{\max_{\alpha_n+1\leq j\leq n}S_j+\max_{\alpha_n+1\leq j\leq n}X_j}}\right]\nonumber
\end{align}
where  $X_n^*:=\max_{1\leq k\leq n}X_k$.

Let us take $\{S_{\alpha_n}^*\leq 0; X^*_{\alpha_n}\leq a_n; S_{\alpha_n}\leq -\beta_n; \max_{1\leq j\leq n-\alpha_n}S_{j+\alpha_n}\leq -\beta_n; \max_{\alpha_n<j\leq n}{X_j}\leq \beta_n-\log n\}$ with $\beta_n\geq \log n$ so that
\[
1+A+(A+B)\alpha_n e^{S^*_{\alpha_n}}+ (A+B) n e^{\max_{\alpha_n+1\leq j\leq n}S_j}+C\alpha_n e^{S_{\alpha_n}^*+X_{\alpha_n}^*}+C n e^{\max_{\alpha_n+1\leq j\leq n}S_j+\max_{\alpha_n+1\leq j\leq n}X_j}
\]
is bounded by $c_4+c_5\alpha_n+c_6\alpha_ne^{a_n}$ where $(c_i)_{i=4,5,6}$ are positive constants. It follows that
\begin{align}\label{lbT}
&\PP(T>n)\\
\geq & (c_4+c_5 \alpha_n+c_6\alpha_ne^{a_n})^{-1}\nonumber\\
 &\quad \times\PP\left(S_{\alpha_n}^*\leq 0; X^*_{\alpha_n}\leq a_n; S_{\alpha_n}\leq -\beta_n; \max_{1\leq j\leq n-\alpha_n}S_{j+\alpha_n}\leq -\beta_n; \max_{\alpha_n<j\leq n}{X_j}\leq \beta_n-\log n\right)\nonumber
\end{align}

By the fact that the increments of the sequence $(S_n)_{n\geq 0}$ are positively associated, one sees that
\begin{align*}
&\PP\left(S_{\alpha_n}^*\leq 0; X^*_{\alpha_n}\leq a_n; S_{\alpha_n}\leq -\beta_n; \max_{1\leq j\leq n-\alpha_n}S_{j+\alpha_n}\leq -\beta_n; \max_{\alpha_n<j\leq n}{X_j}\leq \beta_n-\log n\right)\\
\geq & \PP\left(S_{\alpha_n}^*\leq 0; X^*_{\alpha_n}\leq a_n; S_{\alpha_n}\leq -\beta_n; \max_{1+\alpha_n\leq j\leq n}(S_{j}-S_{\alpha_n})\leq 0; \max_{\alpha_n<j\leq n}{X_j}\leq \beta_n-\log n\right)\\
\geq & \PP(S^*_{\alpha_n}\leq 0)\PP(X_{\alpha_n}^*\leq a_n)\PP(S_{\alpha_n}\leq -\beta_n)\PP(S_{n-\alpha_n}^*\leq0)\PP(X_n^*\leq \beta_n-\log n).
\end{align*}
Let $\beta_n=\alpha_n^Hl(\alpha_n)$ where $\alpha_n=\lfloor (\beta\log n)^{\frac{1}{H-\varepsilon}}\rfloor$ with $\beta>2$ and any $\varepsilon\in(0,H)$ so that for $n$ large enough
\[
\beta_n\geq \beta/2 \log n.
\]
Consequently, by \eqref{taildistribution}, for some $\alpha>1$ and $c_7>0$,
\[
\PP\Big(X_n^*>\beta_n-\log n\Big)\leq n\PP\Big(X_1\geq \Big(\frac{\beta}{2}-1\Big)\log n\Big)\leq n e^{-c_7(\frac{\beta}{2}-1)^\alpha(\log n)^\alpha}=e^{-\Theta(1)(\log n)^\alpha}.
\]
Take $a_n=\big(\frac{1}{c_7}\log(2\alpha_n))^{1/\alpha}$ such that
\[
\PP(X_{\alpha_n}^*>a_n)\leq \alpha_n\PP(X_1>a_n)\leq \alpha_n e^{-c_7 a_n^\alpha}= \frac{1}{2}.
\]
Now by remarking that $\PP(S_{\alpha_n}\leq -\beta_n) = \PP(\frac{S_{\alpha_n}}{\alpha_n^Hl(\alpha_n)} \leq -1)$ converges to $\PP(W(1)\leq -1)>0$ and by applying Theorem 4 in \cite{AGPP}, there exists some constant $c>0$ such that for $n$ large enough
\[
\PP(T>n)\geq \frac{n^{-(1-H)}}{(\log n)^c}\ \sqrt{\ell(n)}.  
\]

\section{Maximal population and total population}
\subsection{Proof of Theorem \ref{main2}}

Let $ \widetilde T(x)$ be the first passage time of the sequence $(S_k)_{k\geq 0}$ above/below the level $x\neq 0$
$$
    \widetilde T(x)
:=
    \left\{\begin{array}{ccc}
    \inf\{k\in\NN;\ S_{k}\geq x\} & \mbox{if} & x>0,\\*
    \inf\{k\in\NN;\ S_{k} \leq x\} & \mbox{if}  & x<0.
\end{array}
\right.
$$
\subsection{Upper bound}
Let us define for every $k\geq 0$, the random variable
\[
W_k:=\frac{Z_k}{\EE^\E[Z_k]}=\frac{Z_k}{\prod_{i=0}^{k-1}f_i'(1)}.
\]
It is well-known that $(W_k)_{k\geq 0}$ is a martingale under the quenched probability. Note that for every $k\geq 0$,
\[
Z_k=W_k \EE^\E[Z_k]=W_k e^{-S_{k}}.
\]
Observe that
\begin{align*}
\PP\left(\max_{0\leq k<T}Z_k\geq N\right)=&\PP\left(\max_{0\leq k<T}Z_k\geq N; T\leq n\right)+\PP\left(\max_{0\leq k<T}Z_k\geq N; T> n\right)
\end{align*}
First, from the upper bound in Theorem \ref{main1}, there exists some constant $c>0$ such that for $n$ large ($n$ will be chosen later)
\begin{equation}\label{star}
\PP\left(\max_{0\leq k<T}Z_k\geq N; T> n\right)\leq \PP(T>n)\leq c n^{-(1-H)}\sqrt{\ell(n)}.
\end{equation}
On the other hand, let $\delta\in(0,1)$,
\begin{align}\label{p2}
&\PP\left(\max_{0\leq k<T}Z_k\geq N; T\leq  n\right)\nonumber\\
\leq &\PP\left(\max_{0\leq k\leq n}W_k\cdot \max_{0\leq k<T}\EE^\E[Z_k]\geq N; T\leq n\right)\nonumber\\
\leq &\PP\left(\max_{0\leq k\leq n}W_k\geq N^\delta\right)+\PP\left(\max_{0\leq k<T}\EE^\E[Z_k]\geq N^{1-\delta}; T\leq n\right)\nonumber\\
\leq & \EE\left[\PP^\E\left(\max_{0\leq k\leq n}W_k\geq N^\delta\right)\right]+\PP\left(\max_{0\leq k\leq n}\EE^\E[Z_k]\geq N^{1-\delta}\right)
\end{align}
Since $(W_k)_{k\geq 0}$ is a martingale under the quenched distribution $\PP^\E$, we get
\begin{equation}\label{starstar}
\PP^\E\left(\max_{0\leq k\leq n}W_k\geq N^\delta\right)\leq \frac{\EE^\E[W_n]}{N^\delta}=\frac{1}{N^\delta}.
\end{equation}
By observing that $\EE^\E[Z_k] = e^{-S_k}$, the second probability in (\ref{p2}) is bounded from above by
\[
\PP\left(\min_{k\leq n}S_k\leq -(1-\delta)\log N\right)
\]
 which is equal, by symmetry of Gaussian variables, to $\PP\left(\max_{k\leq n}S_k\geq (1-\delta)\log N\right)$. 
Applying the maximal inequality in Proposition 2.2 in \cite{KL} implies that
\begin{eqnarray*}
\PP\left(\max_{k\leq n}S_k\geq (1-\delta)\log N\right) & \leq & 2 \PP\left(S_n\geq (1-\delta) \log N\right)\\
& =& 2  \PP\left(\sigma_n^2\ X_1\geq (1-\delta) \log N\right)\\
&\leq & 2 \exp\left(-\frac{(1-\delta)^2 (\log N)^2}{2\sigma_n^2}\right)
\end{eqnarray*}
where $\sigma_n^2:=n^{2H}\ell(n)$ is the variance of $S_n$ by \eqref{eqn:lrd}.

Let us choose $n= \sup\{k; \sigma_k\leq (\log N) (\log\log N)^{-\frac{q}{2}}\}$ with $q>1$. Then,
\begin{eqnarray}\label{startrek}
\PP\left(\max_{0\leq k\leq n}\EE^\E[Z_k]\geq N^{1-\delta}\right)\leq 2 \exp\left(-\frac{(1-\delta)^2 (\log\log N)^q}{2}\right)
\end{eqnarray}
The upper bound follows by gathering $(\ref{star})$, $(\ref{starstar})$ and $(\ref{startrek})$.
\subsection{Lower bound}
On the other hand, for the lower bound, we take $\widetilde{T}(-x)$ and $\widetilde{T}(y)$ for certain $x,y>0$. Then,
\begin{align*}
\PP\left(\max_{0\leq k<T}Z_k\geq N\right)\geq& \PP\left(Z_{\widetilde{T}(-x)}\geq N; \widetilde{T}(-x)<\widetilde{T}(y)\leq n\right)\\
\geq & \PP\left(W_{\widetilde{T}(-x)}\times \underbrace{\EE^\E[Z_{\widetilde{T}(-x)}]}_{e^{-S_{\widetilde{T}(-x)}}}\geq N; \widetilde{T}(-x)<\widetilde{T}(y)\leq n\right)
\end{align*}
We will take $x\geq \log(2N)$ so that $\EE^\E[Z_{\widetilde{T}(-x)}]=e^{-S_{\widetilde{T}(-x)}}\geq 2N$ and obtain by Paley-Zygmund inequality that
\begin{align*}
\PP\left(\max_{0\leq k<T}Z_k\geq N\right)\geq&\PP\left(W_{\widetilde{T}(-x)}\geq 1/2; \widetilde{T}(-x)<\widetilde{T}(y)\leq n\right)\\
\geq & \EE\left[\PP^\E\left(W_{\widetilde{T}(-x)}\geq 1/2\EE^\E[W_{\widetilde{T}(-x)}]\right); \widetilde{T}(-x)<\widetilde{T}(y)\leq n\right]\\
\geq & \EE\left[\frac{1}{4}\frac{\EE^\E[W_{\widetilde{T}(-x)}]^2}{\EE^\E[W^2_{\widetilde{T}(-x)}]}; \widetilde{T}(-x)<\widetilde{T}(y)\leq n\right]\\
=& \frac{1}{4}\EE\left[\frac{1}{\EE^\E[W^2_{\widetilde{T}(-x)}]}; \widetilde{T}(-x)<\widetilde{T}(y)\leq n\right]
\end{align*}
As $(W_k)_{k\geq 0}$ is a martingale, the following equality holds
\[
\EE^\E[W_k^2]=\EE^\E[W_{k-1}^2]+\frac{\sigma^2(f_{k-1})\EE^\E[Z_{k-1}]}{(\EE^\E[Z_k])^2}
\]
where $\EE^\E[Z_k]=\prod_{i=0}^{k-1}f_i'(1)=e^{-S_k}$ and $\sigma^2(f_j)=f_j''(1)+f_j'(1)-(f_j'(1))^2$. 
It follows that
\begin{align*}
\EE^\E[W_n^2]=1+\sum_{j=1}^{n}\frac{\sigma^2(f_{j-1})\EE^\E[Z_{j-1}]}{(\EE^\E[Z_j])^2}=&1+\sum_{k=0}^{n-1}\sigma^2(f_k)e^{S_{k+1}+X_{k+1}}\\
\leq &1+A+(A+B)\sum_{k=1}^n e^{S_k}+C\sum_{k=1}^{n}e^{S_{k}+X_{k}}\hspace{1cm}\mbox{\rm from Assumption \bf (A)}
\end{align*}
Thus,
\[
\PP\left(\max_{0\leq k<T}Z_k\geq N\right)\geq \frac{1}{4}\EE\left[\frac{1}{1+A+(A+B)\sum_{k=1}^{\widetilde{T}(-x)} e^{S_k}+C\sum_{k=1}^{\widetilde{T}(-x)}e^{S_{k}+X_{k}}}; \widetilde{T}(-x)<\widetilde{T}(y)\leq n\right]
\]
It is enough to bound from below the following expectation (since $S_0=0$)
$$\EE\left[\frac{1}{\sum_{k=0}^{\widetilde{T}(-x)} e^{S_k}+\sum_{k=1}^{\widetilde{T}(-x)}e^{S_{k}+X_{k}}}; \widetilde{T}(-x)<\widetilde{T}(y)\leq n\right]
$$
Let $\varepsilon>0$. Let us consider the set $\mathcal G_N$ defined by:
$$
    {\mathcal G}_N
:=
    \mathcal G_N^{(1)}\cap \mathcal{G}_N^{(2)} \cap \mathcal{G}_N^{(3)},
$$
with
\begin{eqnarray*}
    \mathcal{G}_N^{(1)}
& := &
    \Big\{\widetilde T(- \log(2N)) < \widetilde T(1)\Big\},
\\
    \mathcal{G}_N^{(2)}
& := &
    \big\{\widetilde T(1)< (\log N)^{\frac {1+\varepsilon}H}\big\},
\\
    \mathcal{G}_N^{(3)}
& := &
    \left\{ \left(\sum_{k=0}^{\widetilde T(- \log(2N))  }e^{S_{k}} +\sum_{k=1}^{\widetilde T(- \log(2N))  }e^{S_{k}+X_k}\right)^{-1} \geq f(N)\right\},
\end{eqnarray*}
where $f(N):= \frac{1}{\kappa (\log \log N)^{3/H}}$ with
$\kappa>0$ determined in \eqref{valuekappa}.
The lower bound will follow from the following lemma.
\begin{lem}\label{LemmaProbaGN}
There exists a function $\tilde L$ that is slowly varying at infinity such that for large $N$,
$$
     \PP\big(\mathcal{G}_N\big)
\geq
    (\log N)^{-\left(\frac{1-H}{H}\right)} {\tilde L}(\log N).
$$
\end{lem}
\noindent Indeed,
\begin{eqnarray*}
\PP\left(\max_{0\leq k<T} Z_k\geq N\right)&\geq & c_8 \EE\left[ {\bf 1}_{\mathcal{G}_N}
 \left(\sum_{k=0}^{\widetilde T(- \log(2N))}e^{S_{k}}+\sum_{k=1}^{\widetilde T(- \log(2N))  }e^{S_{k}+X_k}\right)^{-1} \right]\\
&\geq & c_8 f(N) \PP\big(\mathcal{G}_N\big)\\
&\geq &    (\log N)^{-\left(\frac{1-H}{H}\right)} L(\log N)
\end{eqnarray*}
where $L$ is a function slowly varying at infinity.\\*
The proof of Lemma \ref{LemmaProbaGN} rests on the two following lemma.
\begin{lem} \label{lem:lemma1}
There exists a function $\tilde L$ slowly varying at infinity such that for large $N$,
$$
    \PP\Big[\mathcal{G}_N^{(1)}\cap \mathcal{G}_N^{(3)}\Big]
\geq
    (\log N)^{-\left(\frac{1-H}{H}\right)} \tilde L(\log N).
$$
\end{lem}
\begin{lem}\label{lem:lemma2}
$$\PP\Big[\big(\mathcal{G}_N^{(2)}\big)^c\Big]= \mathcal{O}\left( (\log N)^{-\frac{(1-H)}{H}(1+\varepsilon)}\ \sqrt{\ell\Big(\lfloor(\log N)^{\frac{1+\varepsilon}{H}}\rfloor\Big)}\right).$$
\end{lem}
\smallskip

\begin{proof}[Proof of Lemma \ref{LemmaProbaGN}]
Note that, by Lemma~\ref{lem:lemma2}, there exists $c_9>0$ such that for every $N$,
\begin{equation}\label{eqProbaGN2}
    \PP\Big[\big(\mathcal{G}_N^{(2)}\big)^c\Big]
\le
 c_9   (\log N)^{-\left(\frac{(1-H)(1+\varepsilon)}{H}\right)}\ell\big((\log N)^{\frac{1+\varepsilon}{H}}\big) . 
\end{equation}
Due to Lemma \ref{lem:lemma1}, for large $N$,
$$
     \PP\big(\mathcal{G}_N\big)
\geq
    \PP\big[\mathcal{G}_N^{(1)}\cap {\mathcal{G}}_N^{(3)}\big]
    -\PP\big[\big(\mathcal{G}_N^{(2)}\big)^c\big]
\geq
    (\log N)^{-\left(\frac{1-H}{H}\right)}\tilde L(\log N),
$$
since the probability of the set
$\big(\mathcal{G}_N^{(2)}\big)^c$ 
is of a lower order by \eqref{eqProbaGN2}.
\end{proof}

\smallskip

\begin{proof}[Proof of Lemma \ref{lem:lemma1}] (see Step 1 in the proof of Lemma 9 in \cite{AGP})
Let $d:=LK$ with $K:=K_N:=\min\{k\in\NN\, :\, k^{2H}\ge 33 (2\log N)^2\}$ and $L:=L_N:=\big\lfloor (\log\log N)^{\frac q{2H}}\big\rfloor$, with $q>H/2(1-H)$ and $q>2H$.
Then,
\begin{eqnarray}
    \PP\big[ \mathcal{G}_N^{(1)} \cap \mathcal{G}_N^{(3)}\big]
& =&
    \PP\bigg[ \widetilde T(-\log(2N)) < \widetilde T(1); \frac{1}{\sum_{k=0}^{\widetilde T(-\log(2N))} e^{S_{k}}+\sum_{k=1}^{\widetilde T(-\log(2N))} e^{S_{k}+X_k}} \geq f(N) \bigg]
\notag \\
& \geq &
    \PP\bigg[ \widetilde T(-\log(2N)) < d < \widetilde T(1); \frac{1}{\sum_{k=0}^{d} e^{S_{k}}+\sum_{k=1}^d e^{S_{k}+X_k} } \geq f(N) \bigg]
\notag \\
&=&
    \PP\bigg[ \widetilde T(1)>d; \frac{1}{\sum_{k=0}^{d} e^{S_{k}}+\sum_{k=1}^{d} e^{S_{k}+X_k}} \geq f(N) \bigg]
\label{eqn:firstst} \\
&-& \PP\bigg[ \widetilde T(-\log(2N)) \geq d; \widetilde T(1)>d; \frac{1}{\sum_{k=0}^{d} e^{S_{k}}+\sum_{k=1}^{d} e^{S_{k}+X_k}} \geq f(N) \bigg]\notag
\end{eqnarray}
We show that the last term in (\ref{eqn:firstst}) is not relevant since it is bounded from above by the probability
\begin{eqnarray}
\PP \Big[ \max_{k=1,\ldots,d} |S_{k}|\leq \log (2N)\Big]  & \leq &   (\log N)^{-\frac{(1-H)}{H} -1}
\end{eqnarray}
using inequality (35) in \cite{AGP}.
For the first term in (\ref{eqn:firstst}), observe that
\begin{eqnarray}
&&
    \PP\bigg[\widetilde T(1)>d; \frac{1}{\sum_{k=0}^d e^{S_{k}}+\sum_{k=1}^d e^{S_{k}+X_k}} \geq f(N) \bigg]
\nonumber\\
& \geq &\PP\left(S_{\alpha_d}^*\leq 0; X^*_{\alpha_d}\leq a_d; S_{\alpha_d}\leq -\beta_d; \max_{1\leq j\leq d-\alpha_d}S_{j+\alpha_d}\leq -\beta_d; \max_{\alpha_d<j\leq d}{X_j}\leq \beta_d-\log d\right)
   \label{InegfofN2a}
\end{eqnarray}
where $\alpha_d, a_d, \beta_d$ are defined in the proof of the lower bound in Theorem 1 (see Section 2).
On the set inside the previous probability (Remark that for large $N$,
$d\leq (\log N)^{\frac 2H}$ and that we take $H-\varepsilon>H/3$):
\begin{equation}\label{valuekappa}
    \sum_{k=0}^{d} e^{S_{k}} +\sum_{k=1}^d e^{S_k+X_k}   \leq c_9+c_{10}\alpha_d+c_{11}\alpha_de^{a_d} \leq
    \kappa(\log\log N)^{3/H}
=
    f(N)^{-1}.
\end{equation}
Using techniques developed in the proof of the lower bound in Theorem 1, the probability (\ref{InegfofN2a}) is bounded from below by
\begin{eqnarray*}
\frac{ d^{-(1-H)}}{ (\log d)^{c}} \sqrt{\ell(d)}
 &\geq &\frac{(\log N)^{-\frac{(1-H)}{H}}}{ L (\log N)}
\end{eqnarray*}
for $N$ large enough. 
 \end{proof}
 \begin{proof}[Proof of Lemma \ref{lem:lemma2}]
 \begin{eqnarray*}
 \PP\Big[\big(\mathcal{G}_N^{(2)}\big)^c\Big] &=& \PP[ \widetilde T(1)\geq (\log N)^{(1+\varepsilon)/H} ]\\
 &\leq &\PP\left(\max_{k=0,\ldots, [(\log N)^{(1+\varepsilon)/H}]} S_{k} \leq 1\right)\\
&=& \mathcal{O}\left( (\log N)^{-\frac{(1-H)}{H}(1+\varepsilon)}\ \sqrt{\ell\Big(\lfloor(\log N)^{\frac{1+\varepsilon}{H}}\rfloor\Big)}\right)
 \end{eqnarray*}
 by applying Theorem 11 of \cite{AGPP}.
 \end{proof}

\end{document}